\documentclass{amsart}
\usepackage[dvips]{color}
\usepackage{graphicx}
\usepackage{epsfig}
\usepackage{tikz}
\usepackage{enumitem}
\usepackage{color}

\usepackage{latexsym}
\usepackage{amssymb}
\usepackage{amsmath}

\usepackage{amsthm}
\usepackage{amscd}
\usepackage{latexsym}
\usepackage{amssymb}
\usepackage{amsmath}
\usepackage{amsthm,comment}
\usepackage{amscd}
\theoremstyle{plain}
\newtheorem{thm}{Theorem}[section]
\newtheorem{lemma}[thm]{Lemma}
\newtheorem{defin}[thm]{Defintion}
\newtheorem{prop}[thm]{Proposition}
\theoremstyle{definition}
\newtheorem{cor}[thm]{Corollary}
\newtheorem{rem}[thm]{Remark}

\newtheorem{ques}{Question}

\newcommand{\future}[1]{{}}  

\usepackage{amsmath,amsfonts,amssymb,amsthm,epsfig,float, color,tikz}
\usepackage[margin=1.5in]{geometry}
\newcommand{\bb}{\mathbb}

\newcommand{\R}{\bb R}

\newcommand{\mb}[1]{\marginpar{\color{blue}\tiny #1 --mb}}
\newcommand{\jc}[1]{\marginpar{\color{brown}\tiny #1 --jc}}
\newcommand{\sh}[1]{\marginpar{\color{red}\tiny #1 --sh}}

\begin{document}
\title{Ergodic Measures in Strata}
\author{Mladen Bestvina}
\author{Jon Chaika}
\author{Sebastian Hensel}
\maketitle

\begin{abstract}
We determine the number of ergodic measures in each stratum of
geodesic laminations on a surface.
\end{abstract}

\section{Introduction}

A fundamental question in dynamics is to understand the ergodic
measures of a system. A more specific instance of this question is to
bound the possible numbers of such ergodic measures in systems of a
given type.  In this article we study this question in the context of
transverse measures on geodesic laminations on hyperbolic surfaces.

\smallskip This problem has already been extensively studied -- both
for laminations (or, equivalently, singular foliations), and in the very
related contexts of translation flows on translation surfaces, smooth
flows on surfaces and interval exchange transformations. In fact,
results in one of these settings often trivially imply results in
related settings and we now present known results in the language of
minimal and filling laminations (though these results may originally
have been stated for one of the other settings).

For minimal and filling orientable laminations on a closed
surface $S$, Katok used the symplectic structure to show that the
number of ergodic measures is at most the genus $g$ of the surface;
see \cite{Katok} and also \cite[Theorem 0.5]{VIET}. Sataev then proved
that this bound is in fact attained \cite{Sat}. Levitt
\cite[Th\'eor\`eme IV.2.1]{levitt} showed that on a surface of genus
$g$ a filling (possibly non-orientable) lamination can have at most $3g-3$ projectively distinct
ergodic measures. Levitt's method was to pass to a branched double
cover where the lamination is orientable and apply Katok's bound.
Later, Lenzhen and Masur \cite[Theorem C]{LM} simultaneosly approximated all ergodic measures for possibly
non-orientable laminations by multicurves and this again proves that
$3g-3$ is an upper bound.
Gabai showed that this bound is attained
\cite[Theorem 9.1]{Gab}.

Laminations on surfaces can be organized into \emph{strata}, according
to whether the lamination is orientable or not, and the
combinatorics of their complementary components. There is an obvious
Euler characteristic constraint connecting the genus of the surface to
the possible combinatorics of the complementary components, and except
for some low-complexity cases, all data satisfying the Euler
characteristic constraint define non-empty strata. See
Section~\ref{sec:curves} and \cite{MS} for details. 

\smallskip Fickenscher showed that in every stratum of orientable
laminations, there is a minimal and filling lamination with $g$
ergodic measures \cite[Section 4]{Fic}. That is, the maximal number of
ergodic measures is constant over all orientable strata of the same
genus.

The main result of this paper shows that different
behaviour happens for general strata: here, different strata of
laminations on the same surface can have different maximal numbers of
ergodic measures, and we exactly classify these numbers from the
combinatorics. This is made precise in the following theorem.

\begin{thm}\label{thm:upper-bound-intro}
	Let $\lambda$ be a filling, measured lamination, and denote by 
	$n_\mathrm{odd}$ the number of complementary regions with an odd number of sides.
	The number of ergodic projective measures on $\lambda$ is at
        most $g+\frac 12 n_{odd}-1$ if $\lambda$ is nonorientable and
        $g+\frac 12 n_{odd} = g$ if $\lambda$ is orientable.

        Moreover, these bounds are optimal: in each nonempty stratum
        there is a lamination realising the above bounds.
\end{thm}

\subsubsection*{Outline of proof}

The proofs of the upper and lower bounds in
Theorem~\ref{thm:upper-bound-intro} are of very different character.

The upper bound morally already follows from Katok's argument. Katok's proof
is in the context of transversely orientable foliations but one can
pass to the oriented branched double cover (as in \cite{levitt}) or consider the Thurston
symplectic form (as we do here), to establish the analogue in the general lamination
case (see Theorem~\ref{thm:dim-ker}).

For the lower bound, the strategy is to first convert the construction
of a lamination with a given number of ergodic measures in a specific stratum
into a topological-combinatorial problem. In fact, the desired
lamination can be constructed from a \emph{triangular configuration} of
curves with a suitable intersection pattern using fairly standard
train track techniques inspired by constructions of minimal nonergodic
laminations; the details are worked out in
Section~\ref{sec:train-tracks}.

The main new technical work lies in constructing the necessary
triangular configurations of curves. Here, we take an inductive approach similar
to the constructions in \cite{MS}. Our construction has certain
``building blocks'' -- explicit configurations of curves on
low-complexity surfaces -- and a number of ``moves'' which allow to
combine these basic blocks in an inductive fashion. Executing this
strategy occupies Section~\ref{sec:curves}.

Finally, in Section~\ref{sec:conf-from-fol} we show that the
appearance of triangular configurations is not an artifact of our
proof, and in fact any minimal lamination has such a corresponding
configuration. The methods of proof are similar to Katok's, and in
particular our approach reproves and strengthens a result of
Lenzhen-Masur \cite{LM}. We remark that, unlike in Lenzhen-Masur's proof we do
not use Teichm\"uller theory.

\subsubsection*{Open questions}

\begin{ques} Let $S$ be a surface and $\tau\subset S$ a 
  train track with all complementary regions polygons or
  once-punctured polygons that fully carries laminations. What is the
  maximal number of mutually singular ergodic measures on a lamination
  fully carried by $\tau$?  The upper bound from Theorem
  \ref{thm:upper-bound-intro} applies, but we don't know if it is
  realized.
\end{ques}

\begin{ques}
  Let $\Lambda$ be a filling lamination on a hyperbolic surface $S$
  and let $\mu$ be a transverse measure on $\Lambda$. Choose a
  transverse interval $I$. We say that $\mu$ is {\it generic} if there
  is a half-leaf $\ell$ such that if $x_1,x_2,\cdots$ are the
  intersection points of $\ell$ with $I$ in order they occur on
  $\ell$, then the counting measure on the first $n$ points converges
  projectively to the measure on $I$ induced by $\mu$.  For example,
  ergodic measures are generic by Birkhoff's ergodic theorem, see Step
  1 in Section 5. There exist generic measures that are not ergodic,
  \cite{CM}. If $\Lambda$ is orientable can it admit more than $g$
  projectively distinct generic measures? Damron-Fickenscher and
  Masur have upper bounds on the number of generic measures including
  showing that the upper bound is $g$ in the minimal, orientable
  stratum \cite{DF}, \cite{MasGen}. Can the number of projective
  classes of generic measures exceed the number in Theorem
  \ref{thm:upper-bound-intro} in a given stratum?
\end{ques}

\begin{ques}
The analogue of Theorem \ref{thm:upper-bound-intro} is open for $GL(2,\mathbb{R})$-orbit closures. For concreteness we restrict to the orientable case. Let $\mathcal{O}$ be a $GL(2,\mathbb{R})$-orbit closure of translation surfaces and $M \in \mathcal{O}$ be a translation surface with minimal vertical flow. Each ergodic measure gives a cohomology class and the tangent space of $\mathcal{O}$ is given by relative cohomology. What is the maximum, over all surfaces in $\mathcal{O}$ with minimal vertical flow, of the number of ergodic measures for the vertical flow whose cohomology class is in the tangent space to $\mathcal{O}$? By Katok's argument, an upper bound is the \emph{rank} of $\mathcal{O}$ in the sense of Wright \cite{CylDef}, but it is unclear if this is attained.  Note, the condition that the cohomology class given by the ergodic measure is in the tangent space is important; for every $g$ there is a rank 1 orbit closure that has a translation surface with a minimal vertical flow that has $g$ ergodic measures, for example by modifying \cite{Vskew}.  %
\end{ques}

\section{Upper Bounds}
In this section we prove the following upper bound on the number of
ergodic measures a lamination can support.
\begin{thm}\label{thm:upper-bound}
  Let $\tau$ be a filling recurrent 3-valent track in $\Sigma$ with
  $n_\mathrm{odd}$ odd complementary regions, and let $\lambda$ be a
  geodesic lamination tightly carried by $\tau$. The number of ergodic
  measures on $\lambda$, up to scale, is at most
  $g+\frac 12 n_{odd}-1$ if $\tau$ is nonorientable and
  $g+\frac 12 n_{odd} = g$ if $\tau$ is orientable. 
\end{thm}
Since every lamination is tightly carried by a filling recurrent
3-valent track whose complementary components and orientability
exactly correspond to those of the lamination, we obtain one half of
our main theorem:
\begin{cor}
  Suppose $\lambda$ is a measured geodesic lamination, and let
  $n_\mathrm{odd}$ be the number of odd complementary regions of
  $\lambda$. Then the number of ergodic measures on $\lambda$ is at
  most $g+\frac 12 n_{odd}-1$ if $\lambda$ is nonorientable and
  $g+\frac 12 n_{odd}=g$ if $\lambda$ is orientable.
\end{cor}

The rest of this section is concerned with the proof of
Theorem~\ref{thm:upper-bound}. Let $\tau$ be a filling recurrent
3-valent train track in $\Sigma$. We denote by $W=W(\tau)$ the vector
space of all real (including negative) solutions to the switch
equations.  Recall that its dimension is $-\chi(\tau)$ if $\tau$ is
nonorientable and it is $1-\chi(\tau)$ if $\tau$ is orientable.

It is equipped
with the Thurston intersection pairing
$$\Omega:W\times W\to \R$$
defined by
\begin{equation}\label{eq:kernel|}\Omega(\alpha,\beta)=\frac 12 \sum_{e\mbox{ right of
  }e'}(\alpha(e)\beta(e')-\beta(e)\alpha(e'))
  \end{equation} where the sum is over
the pairs of edges $e,e'$ incident at a vertex and forming an illegal
turn so that in the natural cyclic order $e$ is to the right of
$e'$. This is an alternating pairing and there is the following
theorem of Bonahon-Wong \cite{BW}. Let $n_{odd}$ and $n_{even}$ denote
the number of odd and even sided complementary components (they are
disks or once punctured disks) and let $g$ be the genus of $\Sigma$.

\begin{thm}[{\cite[Theorem~26]{BW}}]\label{thm:bw}
  The kernel of the form $\Omega$ has dimension
  \begin{itemize}
  \item $n_{even}$ if $\tau$ is nonorientable,
  \item $n_{even}-1$ if $\tau$ is orientable.
  \end{itemize}
  In both cases the kernel is spanned by the elements of $W$
  corresponding to each even sided region that alternate weights on
  the sides between
  $1$ and $-1$; in the nonorientable case this is a basis, and in the
  orientable case there is one redundancy (the sum with suitable signs
  is 0).
\end{thm}

\begin{lemma}\label{negative}
  Every nonzero element of $Ker(\Omega)$ assigns a negative value to
  at least one branch.
\end{lemma}

\begin{proof}
  Suppose there are nonzero elements of the kernel that assign
  nonnegative values to all branches. Then there are also such
  elements that are rational, because the equation defining the kernel
  has only rational coefficients. Hence there are also elements that
  are integral. Thus we can represent an element of the kernel by a
  multicurve $\alpha$ carried by $\tau$.

  First assume that $\tau$ is orientable. Then the
  Thurston pairing is given by the algebraic intersection number. We
  will construct a curve $\beta$ carried by $\tau$ so that
  $\alpha\cdot\beta\neq 0$. We will arrange that all intersections
  between $\alpha$ and $\beta$ have the same sign. We work in the
  fibered neighborhood of $\tau$. Choose a branch and
  start drawing $\beta$ in the direction of the orientation and to the
  left of all strands of $\alpha$. Whenever a switch is reached where
  there is a choice of a left or right turn, choose either turn and
  continue drawing on the left of all strands of $\alpha$. Eventually
  the curve $\beta$ closes up. If there are strands of $\alpha$ making
  a left turn somewhere, $\beta$ can be constructed so that it
  intersects $\alpha$. Otherwise, reverse the roles of left and
  right. Since $\tau$ is recurrent and not a circle, $\alpha$ must
  make a turn somewhere.

  If $\tau$ is nonorientable, lift $\alpha$ to a multicurve
  $\tilde\alpha$ in the orientation double cover $\tilde\tau\to\tau$.
  The weights of $\tilde\alpha$ are nonnegative, so by the above
  paragraph there is a curve $\beta$ carried by $\tilde\tau$ so that
  $\Omega(\tilde\alpha,\beta)\neq 0$. If $t:\tilde\tau\to\tilde\tau$
  is the covering involution, then
  $\Omega(\tilde\alpha,\beta)=\Omega(\tilde\alpha,t(\beta))$ so
  $\Omega(\tilde\alpha,\beta+t(\beta))\neq 0$. It follows that
  $\alpha$ is not in the kernel of $\Omega$ since it pairs
  nontrivially with the projection of $\beta$. Note that the
  projection of $\beta$ is not homotopically trivial because it pairs
  non-trivially with $\alpha$.
\end{proof}

We say that a geodesic lamination $\lambda$ carried by $\tau$ is {\it
  tightly carried} if some (hence every) splitting sequence from
$\tau$ to $\lambda$ contains no central splits.
We are grateful to James Farre for explaining the following theorem to
us. See also \cite[Proposition 8.1]{Yocsurv}.

\begin{thm}\label{thm:dim-ker}
  Let $\tau$ be a filling recurrent 3-valent track in $\Sigma$ and let
  $\lambda$ be a geodesic lamination tightly carried by $\tau$. The
  number of ergodic measures on $\lambda$, up to scale, is at most
  $$\frac 12(\dim W-\dim Ker(\Omega)).$$
\end{thm}

\begin{proof} (James Farre)
  Let $W(\lambda)\subset W$ be the cone of weights coming from
  transverse measures on $\lambda$. This is a simplicial cone and its
  dimension is equal to the number of ergodic transverse measures on
  $\lambda$. Fix a point $w_0\in W(\lambda)$ in
  the relative interior of $W(\lambda)$. The tangent space of $W$ at
  $w_0$ can be identified with $W$. We will argue that the tangent
  space to $W(\lambda)$ at $w_0$ and $Ker(\Omega)$ intersect
  trivially. This will imply that the tangent space to $W(\lambda)$
  injects and is
  a Lagrangian in the quotient $W/Ker(\Omega)$ and so has dimension at
  most half of that of $W/Ker(\Omega)$, proving the theorem.

  To prove the claim by contradiction, suppose there is some nonzero
  vector $v\in Ker(\Omega)$ so that $w_0+v\in W(\lambda)$. Thus $v$ is
  some linear combination of the elements of $Ker(\Omega)$
  corresponding to the even sided regions. When $\tau$ is
  (noncentrally) split to $\tau'$, there
  is a natural identification $W(\tau)\cong W(\tau')$ that identifies
  the cones of measures transverse to $\lambda$ and the kernels of the
  forms. Moreover, the elements of the kernel corresponding to even
  sided region also correspond and so do the coefficients in a linear
  combination of these elements.
  
  Now consider a splitting sequence from $\tau$ towards $\lambda$. 
  Since there are only finitely many combinatorial types
  of train tracks there will be only finitely many weight vectors
  along this sequence induced by $v$. Each has a negative weight by
  Lemma \ref{negative}, so
  there is some $\epsilon>0$ so that in each track in the sequence
  some branch gets induced weight $<-\epsilon$. On the other hand, the
  weights induced by $w_0$ will converge to 0. This means that
  eventually $w_0+v$ will assign a negative weight to some branch,
  contradiction. 
\end{proof}

We are now ready to finish the proof of Theorem~\ref{thm:upper-bound}.
Namely, starting with the train track $\tau$, and attaching disks to
the boundaries of each of the $n_{even}+n_{odd}$ (possibly punctured)
complementary regions, we obtain a closed surface of genus $g$. An
Euler characteristic count then yields
\[ \chi(\tau)+n_{even}+n_{odd}=2-2g. \] If $\tau$ is nonorientable,
recall that $\dim W(\tau) = - \chi(\tau)$, and that by
Theorem~\ref{thm:bw} we have $\dim \ker(\Omega) = n_\mathrm{even}$.
By Theorem~\ref{thm:dim-ker} the number $k$ of ergodic measures on
$\lambda$ is then at most
\[ \frac{1}{2}(\dim W(\tau) - \dim\mathrm{ker}(\Omega))
  = \frac{1}{2}(- \chi(\tau) - n_\mathrm{even})
  = \frac{1}{2}(2g-2 + n_\mathrm{odd}). \]

If $\tau$ is orientable,
recall that $\dim W(\tau) = 1 - \chi(\tau)$, and that by
Theorem~\ref{thm:bw} we have $\dim \ker(\Omega) = n_\mathrm{even} - 1$.
By Theorem~\ref{thm:dim-ker} the number $k$ of ergodic measures on
$\lambda$ is then at most
\[ \frac{1}{2}(\dim W(\tau) - \dim\mathrm{ker}(\Omega))
  = \frac{1}{2}(2 - \chi(\tau) - n_\mathrm{even})
  = \frac{1}{2}(2g + n_\mathrm{odd}). \]

This shows the claims in Theorem~\ref{thm:upper-bound}. 
\section{Triangular Configurations}
\label{sec:curves}
We will be concerned with collections of pairs of multicurves $\left(
(a_i, b_i) \right)_{i=1}^k$ on a closed oriented surface $S$. We
say that such a collection is \emph{triangular} if
\[ I(a_i, b_j) = I(a_i,a_j)=I(b_i,b_j)=0 \quad \quad j > i. \] and
\[ I(a_i,b_i)>0\] where $I$ is the geometric intersection number. In
addition, we will assume that all curves involved intersect
tranversely, are distinct from each other and fill $S$. In this
section we will allow parallel curves as well as bigon complementary
regions; in later application every such bigon or annulus will be
required to contain a marked point, so that the resulting curves are
in minimal position \emph{on the corresponding punctured surface}.

Thus the collections
$\cup_i a_i$ and $\cup_i b_i$ are multicurves, partitioned into $k$ subsets.

We say that the collection is
\emph{orientable} if (for a suitable choice of orientations on the
curves) all intersection points have the same sign.

By \emph{regions} we mean complementary components of $\bigcup a_i\cup
b_i$. These are (topological) polygons
with an even number of sides (alternating between arcs of $\cup a_i$ and
$\cup b_i$). Since a region with $2m$ sides with generate a $m$--prong
singularity as we shall see later, we call it a \emph{$m$--prong region}, and say that
it is \emph{odd} if $m$ is odd, and \emph{even} otherwise. That is, an
even region is a complementary polygon whose number of sides is
divisible by $4$. We denote by $n_{\mathrm{odd}}$ the number of odd
regions of the collection, and by $n_{\mathrm{even}}$ the number of
even regions. When $m=1$ we will ultimately have a puncture in the
region, but for now we work with the closed surface $S$.

\smallskip For any collection of pairs of curves $((a_i,b_i))$,
suppose that there are $N$ complementary regions $R_1, \ldots, R_N$,
and for each $j$ suppose that $R_j$ is a $m_j$--prong region.
Then the $m_j$ satisfy the Euler characteristic constraint:
\[ \sum \frac{2-m_i}2 = \chi(S)=2-2g. \] We say that a list
$m_1, \ldots, m_N$ of numbers satisfies the \emph{Euler characterisic
  constraint} on a genus $g$ surface, if the above is true. It will be
convenient to omit any 2's from the list of these numbers, as they
represent 2-prong singularities, so they are not singular at all. By
the {\it signature} of the collection $((a_i,b_i))$ we mean the pair
$(R,s)$ where $R$ is the unordered list of numbers representing prongs
in the complementary regions with 2's omitted and $s=\pm$ with $+$
representing orientability.

\smallskip In light of Theorem~\ref{thm:upper-bound}, 
for a triangular, non-orientable collection of $k$ multicurves
on a genus $g$ surface we say that it is \emph{optimal} if the upper
bound in the Thurston form is achived, i.e. if
$$k=g-1+n_{odd}/2.$$
Similarly, for a triangular, orientable collection of $k$ multicurves on a
genus $g$ surface we say that it is optimal if
$$k=g+n_{odd}/2.$$

The goal of this section is to prove the following result.
\begin{thm}\label{thm:existence-triangular-patterns}
  Let $S$ be a surface of genus $g\geq 0$. Suppose that
  \[ R=(m_1, \ldots, m_N) \] is a list of numbers and $s=\pm$.
  Assume that
  \begin{enumerate}[label = (\roman*)]
  \item $\sum\frac{2-m_j}{2}=2-2g$,
    \item if some $m_j$ is odd then $s=-$,
  \item if $g=1$, the input is not $R=(1, 3)$, nor $R=\emptyset$ with
    $s=-$, and
  \item if $g=2$, the input is not $R=(3, 5)$, nor $R=(6)$ with $s=-$.
  \end{enumerate}
  Then there is an optimal triangular configuration on $S$ whose
  signature is $(R,s)$.
\end{thm}

\begin{rem}
The excluded signatures are impossible to realize, see \cite{MS}. 
There is also a direct way to see this. The assumptions that the
configurations are triangular or optimal don't play any role below. We
only assume we have two transverse multicurves $a$ and $b$ 
that together fill
and we consider the complementary regions.
\begin{itemize}[leftmargin=5.5mm]
  \item $R=\emptyset$: all curves are parallel to two distinct curves
    and on the torus two curves with any orientations intersect with
    the same sign, so the configuration is orientable.
    \item $R=(1,3)$. Place a puncture in the 1-prong region. If two
      $a$ or two $b$ curves cobound an annulus that doesn't contain
      the puncture, then all regions in the annulus are quadrilaterals
      and one of the curves can be removed without changing region
      types. Similarly, if an $a$ and a $b$ curve cobound an annulus
      without the puncture, then the annulus cannot be intersected by
      any other curves, which is a contradiction. Now choose a
      complete hyperbolic metric on the punctured torus and realize
      all curves by geodesics. The
      hyperelliptic involution rotates by $\pi$ around the puncture
      and preserves all simple closed geodesics, giving a
      contradiction at the 1-prong region.
      \item $R=(3,5)$. We may again assume that there are no parallel
        curves and realize them as geodesics with respect to a
        hyperbolic metric. The hyperelliptic involution has to preserve
        both 3- and 5-prong regions, so it rotates the polygons thus
        mapping some $a$-curve to a $b$-curve. This contradicts the
        fact that it preserves all curves. 
        \item $R=((6),-)$. We can again assume that there are no
          parallel curves. First note that no given curve can be
          separating. Indeed, the complementary components would be
          tori with one boundary component, and it is not possible
          to tile these surfaces by quadrilaterals (for example by
          Poincar\'e-Hopf), so there would have to be a region which
          is not 2-prong in each component, contradiction. It follows that the multicurves $a$ and
          $b$ contain one or two curves each. We now consider two
          cases.

          {\it Case 1. There exist an $a$-curve $\alpha$ and a
            $b$-curve $\beta$ that
            have intersection points of both signs.}

          Cutting along $\alpha$ yields a torus with two boundary
          components and $\beta$ yields arcs joining boundary
          components. For each boundary component there is at least
          one arc with both endpoints on it. Such arcs can be
          separating or nonseparating. There are two possibilities for
          separating arcs, pictured in Figure \ref{6}; the second kind
          implies the existence of the first at the other boundary
          component. This leads to a contradiction since the resulting
          1-prong region, even when subdivided by other curves, will
          contain a 1-prong region.

          \begin{figure}
\includegraphics[scale=0.5]{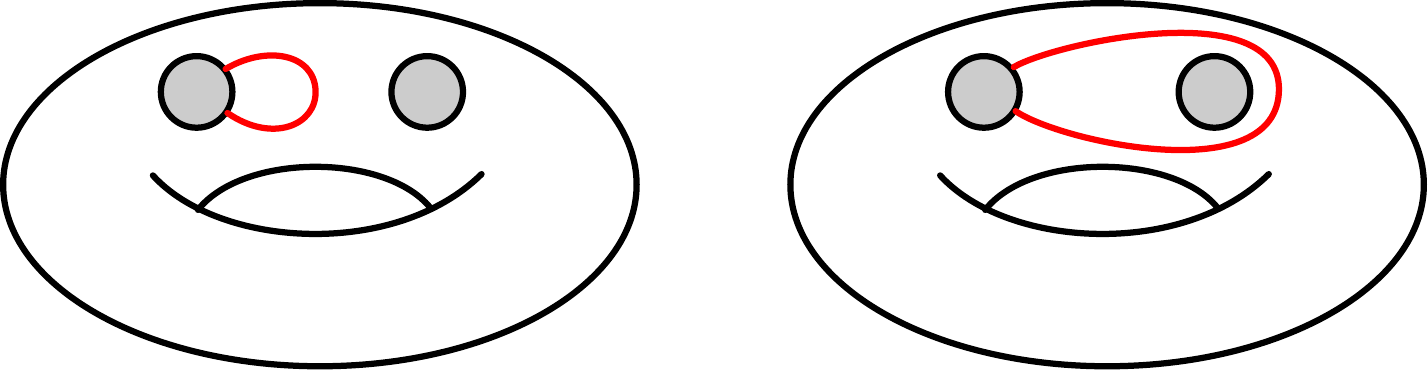}
\caption{Two types of separating arcs.}
\label{6}
\end{figure}

          If all the arcs coming from $\beta$ are nonseparating, we
          have the situaton as in Figure \ref{6a}. There will have to
          be a non-2-prong region in each complementary component,
          contradiction. 

          \begin{figure}
\includegraphics[scale=0.5]{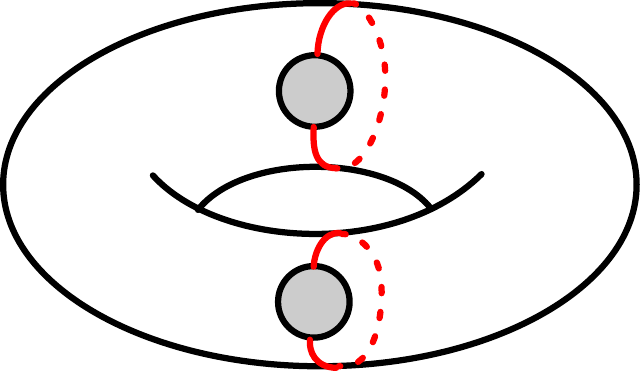}
\caption{Nonseparating arcs.}
\label{6a}
\end{figure}

         {\it Case 2. Any two curves have all intersection points of
           the same sign.}

         If $a$ or $b$ consist of a single curve then the collection
         is orientable. So we may assume $a=\{\alpha_1,\alpha_2\}$ and
         $b=\{\beta_1,\beta_2\}$, and in fact $\alpha_i\cap\beta_j\neq
         0$ for all $i,j$. We can orient $\alpha_1,\alpha_2,\beta_1$
         so that all intersections have the same sign, but then
         $\beta_2$ intersects $\alpha_1$ and $\alpha_2$ in opposite
         signs. Cut along $\alpha_1$ and $\alpha_2$, producing a
         sphere with 4 boundary components. $\beta_1,\beta_2$ yield
         arcs that connect these boundary components in a cyclic
         order. The complement of these 4 arcs consists of two 4-prong
         regions and therefore there will be at least two regions that
         are not 2-prong. \qed
\end{itemize}
\end{rem}

The proof of Theorem \ref{thm:existence-triangular-patterns} will
occupy the rest of this section, and has an inductive nature. There
are various moves, which modify a collection of pairs of curves (and
possibly the surface), which allow us to reduce complexity of the list
which is to be constructed. In genus $0$ we consider certain building
blocks, which form the basis of the induction. When attempting to
realize a given signature by an optimal triangular configuration, we
will always assume it satisfies (i) and (ii) above.

\subsection{Spheres}

We start by investigating genus $0$. Here, any collection is
non-orientable, and the Euler characteristic constraint is that the
sum of $\frac{2-m_i}2$ over all $m_i$-prongs should be $2$.

\bigskip Figure \ref{blocks} displays our two building blocks with
$k=2$. One is $(1^5,3)$ (a shorthand for $(1,1,1,1,1,3)$) and the
other is $(1^6,4)$. For later reference, we record the crossing types; following the convention that the arc between the first two listed regions is always red, and the regions are traversed clockwise. With this convention the type is defined up to swapping the first two with the last two entries.
\begin{itemize} 
	\item 1123, 1232, 2113, 1321, 3212, 2311. %

	For later use, we emphasise that the crossings of type 2113 and 2311/3211 only share the region of type 3.
	\item 2412, 4121, 1214, 4221, 2124, 1241, 2113, 1242

	Similar to above, we note that the crossings 1214 and 2114/1241 only share the region of type 3.
\end{itemize}

\begin{figure}[h]
\includegraphics[scale=0.5]{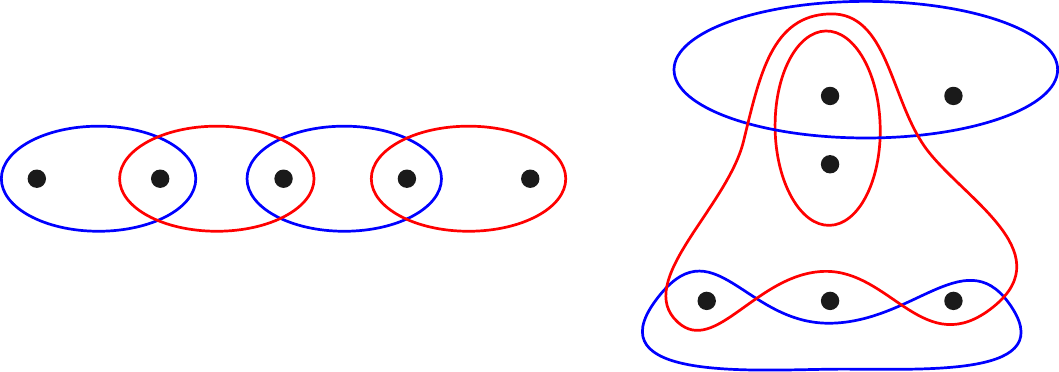}
\caption{The building blocks.}
\label{blocks}
\end{figure}

\subsubsection*{Surgery at a crossing in genus 0}
We will realize all other signatures in genus 0 by connect summing the
building blocks. Here all multicurves $a_i,b_i$ will be single
curves. We start by describing the basic operation.
We consider two triangular configurations $((a_i,b_i))$ with $k_1$
pairs of multicurves
and $((c_j,d_j))$ with $k_2$ pairs of multicurves on spheres. Choose a crossing $p$
between $a_i$ and $b_i$ on one and $q$ between $c_j$ and $d_j$ on the
other. Cut a disk around each crossing and glue the two boundary
components so that the arcs $a_i$ and $c_j$ glue into one curve $e$, and $b_i$
and $d_j$ into another curve $f$. 
We say that they are obtained by \emph{surgery at the crossings $p,q$}.
\begin{lemma}[Crossing Surgery]
  Let $((a_i,b_i))$ be a triangular configuration with $k_1$ pairs of multicurves
  and $((c_i,d_i))$ a triangular configuration with $k_2$ pairs of multicurves on
  spheres. Let $p$ be a crossing of $a_i, b_i$ and $q$ a crossing of
  $c_i, d_i$. Assume that the four regions around $p$ are all
  distinct, and the same is true at $q$.
  \begin{enumerate}[label = (\roman*)]
  \item The collection obtained by surgery at $p,q$ as above is triangular and
    has $k=k_1+k_2-1$ pairs of curves.
  \item If $((a_i,b_i))$ and $((c_i,d_i))$ are both 
    optimal, then the same is true for the surgered collection if the
    total number of odd regions in the connected sum is 4 less than
    the sum of the numbers of odd regions in the original
    configurations; equivalently, in every pair of matched regions at
    least one is odd.
      \end{enumerate}
\end{lemma}

\begin{proof}
  For (i), note that $I(a_i,b_i)>1$ since $S=S^2$, so
  $I(e,f)>0$. Now order the collection as
  $$(a_1,b_1),\cdots,(a_{i-1},b_{i-1}),(c_1,d_1),\cdots,(c_{j-1},d_{j-1}),(e,f),$$$$(a_{i+1},b_{i+1}),\cdots,(a_{k_1},b_{k_1}),(c_{j+1}d_{j+1}),\cdots,(c_{k_2},d_{k_2})$$
  For (ii),
  first observe that the complementary polygons of the
  surgered collection are either complementary polygons of the
  original collections (if the region did not touch $p, q$), or is
  obtained by joining a region of $((a_i, b_i))$ touching $p$ with a
  region of $((c_i, d_i))$ touching $q$. There are exactly four
  regions of the connected sum which are of this type. If the original
  polygons had $2m_1$ respectively $2m_2$ sides, the resulting polygon
  will have $2m_1 + 2m_2 - 2$ sides. %
  This implies that the region is odd exactly if either both joined
  regions are even or both joined regions are odd.
  
  \smallskip Since any collection on a sphere is non-orientable, for
  an optimal configuration with $k$ curves, we require that
  \[ k=n_{odd}/2 - 1. \] In other words, there need to be
  $2(k+1) = 2(k_1+k_2)$ odd regions in the surgered collection. Since
  the original collections are optimal, they have $2(k_1+1)$
  respectively $2(k_2+1)$ odd regions. Hence, during the surgery we
  need to ``lose'' four odd regions. The loss of one must occur in
  every matched pair, forcing one (or both) of the numbers to be odd.
\end{proof}

The general signature in genus 0 will be realized by connect
summing the building blocks in a linear fashion, so that two of the
spheres are connect-summed to one other sphere and all others are
connect-summed to two spheres. It is convenient to explain this
construction in stages.
  
\begin{prop}\label{prop:basic-odd}
  For every odd $m$ the signature $(m,1^{m+2})$ is realized by an optimal configuration. Furthermore, we can choose the configuration so that there are crossings of type $112m$ and $211m$ 
  where the region of type $m$ is the only region these crossings have in common.
\end{prop}

\begin{proof}
  When $m=3$ this is a building block. For $m=5$ do surgery on two
  copies of the building block matching 1123 with 2113, resulting in
  2125 (see Figure \ref{surgery}). On the first building block, the crossing of type 2113 has
  turned into 2115, while on the second block the crossing 1123 has
  turned into 1125. They are still independent in the sense that the
  region 5 is the only region in common between those crossings.

  \begin{figure}[h]
\includegraphics[scale=0.7]{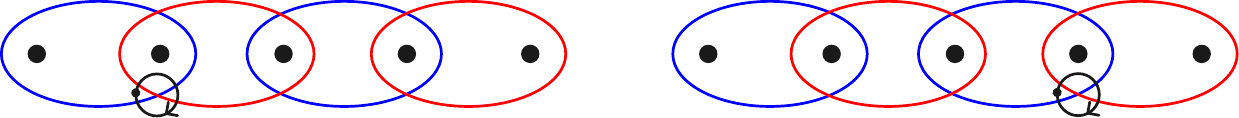}
\caption{Performing surgery by removing disks bounded by black curves
  and gluing along the curves, matching the dots and orientations,
  produces an optimal configuration with signature $(5,1^7)$.}
\label{surgery}
\end{figure}

  For larger $m$ we can now argue inductively: by matching 2113 in the
  building block to the the crossings of type $112(m-2)$ of the
  previous step we obtain a crossing of type $212m$. The crossing of
  type $1123$ of the building block turns into $112m$, while the
  crossing of type $211(m-2)$ of the previous step turns into $211m$,
  and they stay independent.
\end{proof}

\begin{prop}\label{prop:basic-even}
  For every even $m$ the signature $(m,1^{m+2})$ is realized by an optimal
  configuration. Furthermore, we can choose the configuration so that
  there are two crossings of type $211m$ and $112m$ where the region
  of type $m$ is the only region these crossings have in
  common.
\end{prop}
  \begin{proof} 
  When $m=4$ this is a building block. For $m=6$ do surgery on the two
  different building blocks matching crossings 2113 (on the $m=3$ side) with 1214 (on the $m=4$ side). The result is a crossing of type 2216. For the same reason as in the previous proof, on the $m=3$ side of the connect sum curve there is now a crossing of type 1126, and on the $m=4$ side there is a crossing of type 2116. They are independent in the sense that the region 6 is the only common region. This satisfies the conclusion.
  
  In the next step we glue 1123 to 2116 to form 2128; the 1126 from before turns into 1128, and on the new $m=3$ side we find 2118. Hence we can inductively continue to find the desired configurations, every time increasing the large even number by 2.
  \end{proof}

  \begin{prop}\label{prop:same-partity}
    Any signature $(n_1, \ldots, n_k,1^{n_1+\cdots+n_k+2})$ where
    the $n_i$ are all odd or all even is realised by an optimal
    configuration. Furthermore there is a crossing of type
    $112n_1$. 
  \end{prop} 
  \begin{proof}
    We begin by taking the corresponding configurations from Proposition~\ref{prop:basic-odd} or~\ref{prop:basic-even} for the $n_i$. Now glue the 
    crossings of type $211n_2$ to $1n_121$. The result is a crossing of type $2n_12n_2$. The result also has independent crossings of the types $112n_2$ and $112n_1$. 
    
    To generate the other configuration, we will continue gluing, never involving the crossing $112n_1$, ensuring the last part of the claim. Namely, in the next step, we can glue $2n_311$ to $112n_2$ to obtain $2n_32n_2$, and (on the $n_3$--side) $211n_3$. We can thus glue $1n_421$ to this crossing and so on.
  \end{proof}
  
  \begin{thm}\label{genus 0}
Any signature satisfying the Euler characteristic condition for
genus 0 is realized by an optimal configuration. %
  \end{thm}

  \begin{proof}
    First note that $1^4$ is realized by two great circles. 
    Configurations with only odd or only even entries $>2$ are handled by Proposition~\ref{prop:same-partity}. Finally, to get configurations of mixed parity, first construct configurations for only the even entries $n_i$ and only the odd entries $m_j$ using Proposition~\ref{prop:same-partity}. Then glue the crossings $112m_1$ and $2n_111$ guaranteed by that proposition.
  \end{proof}

\subsection{Moves in genus $>0$}
In this subsection we present a collection of moves on triangular
configurations on surfaces and study their effect on orientability and
optimality.

\subsubsection*{Adding Handles}
Here, we present various moves of similar flavor. In each,
we begin with a triangular configuration $((a_i,b_i))$ on a genus $g$
surface. In the first we choose a complementary region $R$ with $2r$
sides and add a handle to the inside of the region (increasing the
genus by $1$), and add a pair of curves to the collection one of which
is the meridian of the handle, and one which follows one of the curves
of the original collection. See Figure~\ref{handle3}. We say that the
new collection is obtained by \emph{adding a handle in a region}
\begin{figure}[h]
    \includegraphics[scale=0.6]{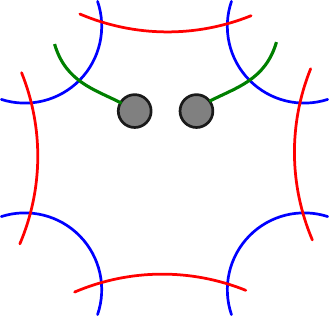}
    \caption{Adding a handle to the same region}
    \label{handle3}
\end{figure}
\begin{lemma}[Adding a handle in a region]\label{lem:adding-handle-one}
  Suppose that $((a_i,b_i))$ is a triangular configuration, and $R$ a
  complementary component.
  \begin{enumerate}[label = (\roman*)]
  \item Every region of the collection obtained by adding a handle is
    either a region of $((a_i,b_i))$ distinct from $R$, or has eight
    more sides than $R$ (and there is exactly one such).
  \item The collection obtained by adding a handle in $R$ is
    orientable if and only if $((a_i,b_i))$ is.
  \item The collection obtained by adding a handle in $R$ is
    optimal if and only if $((a_i,b_i))$ is.
  \end{enumerate}
\end{lemma}
\begin{proof}
  The first claim is clear from construction. For the second claim
  choose the orientation on the new curve in the same way as the curve
  it follows, and then choose the orientation on the meridian
  accordingly.

  The last claim follows since adding a handle in $R$ increases both
  the genus and the number of curves in the collection by $1$, while
  having no effect on the number of odd regions.
\end{proof}
This allows the following modification of signatures.
\begin{prop}[Subtract 4]\label{prop:handle-plus-4}
  Suppose that $((n,m_1,\cdots,m_k),s)$ is a signature realized by an
  optimal collection of multicurves in genus $g-1$. Then the signature
  $((n+4,m_1,\cdots,m_k),s)$ is realized by an optimal collection in
  genus $g$.
\end{prop}
\begin{proof}
  Realise the former signature in genus $g-1$, and add a handle to the
  $n$-prong region using Lemma~\ref{lem:adding-handle-one}.
\end{proof}

Similarly to the above, we can add a handle joining two different
regions.  Namely, choose distinct complementary regions $R_1, R_2$ with $2r_i$
sides adjacent to the same curve $\alpha$, which is component of some
$a_i$ or $b_i$. We say such regions are
\emph{coadjacent}. Now perform a surgery by cutting out a disk from
both $R_i$ and glue a cylinder along the resulting boundary
components, preserving the orientability of the surface
(increasing the genus by $1$), and add a pair of curves to the
collection one of which is the meridian of the handle, and one which
follows $\alpha$. See Figure~\ref{handle}.
  \begin{figure}[h]
    \includegraphics[scale=0.7]{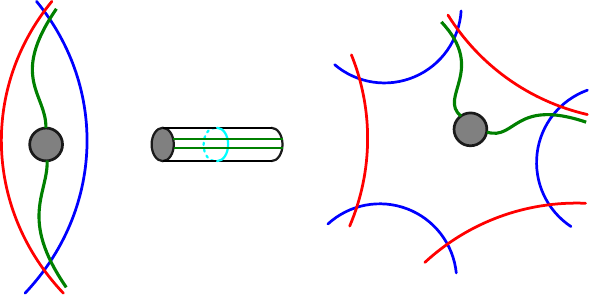}
    \caption{Joining two regions with a handle}
    \label{handle}
  \end{figure}
  We say that the new collection is obtained by \emph{joining coadjacent regions
  $R_1, R_2$ with a handle}. This produces a new triangular
  collection. If $\alpha$ belongs to some $a_i$, call the meridian
  curve $b$ and let $a$ be the curve that follows $\alpha$, and place
  $(a,b)$ at the end of the ordered list of pairs of multicurves. If
  $\alpha$ belongs to some $b_i$, let $a$ be the meridian and $b$ the
  curve that follows $\alpha$, and place $(a,b)$ at the beginning.

\begin{lemma}[Joining regions with a handle]\label{lem:adding-handle-coadjacent}
  Suppose that $((a_i,b_i))$ is a triangular configuration, and $R_1, R_2$ 
  complementary components which are coadjacent.
  \begin{enumerate}[label = (\roman*)]
  \item Every region of the collection obtained by adding a handle is
    either a region of $((a_i,b_i))$ distinct from $R_1, R_2$, or has four
    more sides than $R_i$ (and there is exactly one such for each $R_i$).
  \item The collection obtained by joining $R_1$ to $R_2$ with a handle is
    orientable if and only if $((a_i,b_i))$ is.
  \item The collection obtained by adding a handle in $R$ is
    optimal if and only if $((a_i,b_i))$ is.
  \end{enumerate}
\end{lemma}
\begin{proof}
  The first claim is clear from construction. For the second claim
  choose the orientation on the new curve in the same way as the curve
  it follows, and then choose the orientation on the meridian
  accordingly.

  The last claim follows since adding a handle in $R$ increases both
  the genus and the number of curves in the collection by $1$, while
  having no effect on the number of odd regions.
\end{proof}
This lemma allows the following modification of signatures.
\begin{prop}[Remove Fours]\label{prop:handle-plus-22}
  \begin{enumerate}[label = (\roman*)]
  \item Suppose that $((m_1,\cdots,m_k),s)$ is a signature realized by an
  optimal collection of multicurves in genus $g-1$. Then the signature
  $((4,4,m_1,\cdots,m_k),s)$ is realized by an optimal collection in
  genus $g$.
\item Suppose that $(1,m_1,\cdots,m_k)$ is a (nonorientable) signature realized by an
  optimal collection of multicurves in genus $g-1$. Then the signature
  $(4,3,m_1,\cdots,m_k)$ is realized by an optimal collection in
  genus $g$.
  \end{enumerate}
\end{prop}
\begin{proof}
  We begin with (i). Here, we start by realising the signature
  $((m_1, \ldots, m_k),s)$ by an optimal configuration in genus
  $g-1$. Choose some component $\alpha$ of some $a_i$, and let $\alpha'$ be
  a curve parallel to $\alpha$ so that the annulus bounded by $\alpha$
  and $\alpha'$ is subdivided into 4-gons by the other curves. Add
  $\alpha'$ to the multicurve $a_i$ and then add a handle to two of
  the 4-gons in the annulus. These are coadjacent by construction and
  the effect is that the 4-gons (i.e. 2-prong regions) are replaced by
  8-gons (i.e. 4-pronged regions).

  For (ii), we argue similary, realising $(1, m_1, \ldots, m_k)$ in genus
  $g-1$, doubling a suitable curve, and joining a $1$--prong region to
  a $2$--prong region.
\end{proof}

\subsubsection*{Adding $1$-prongs}
Finally, we need a move that does not change the genus.
\begin{prop}[Remove 3,1]\label{prop:3-to-331}
  Suppose that $((1,m_1,\cdots,m_k),-)$ is a signature realized by an
  optimal collection of multicurves in genus $g$. Then the signature
  $((3,1,1,m_1,\cdots,m_k),-)$ is realized by an optimal collection in
  genus $g$.
                  \end{prop}

\begin{proof}
          Consider the $1$--prong region. Double one of the curves to
          decompose it into a rectangle and bigon. Now add a new curve
          pair which intersects one of the doubled curves and generates 2 bigons; compare
          Figure~\ref{1to311}.

          \begin{figure}[ht]
            \includegraphics[scale=0.8]{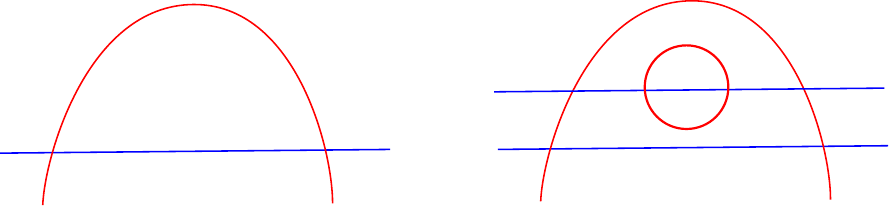}
            \caption{Changing a $1$--prong region to three regions with $3,1,1$ prongs.}
            \label{1to311}
          \end{figure}
          
        \end{proof}

\subsubsection*{Maximal Cases}
Finally, we need to realise some cases by hand in every genus:
\begin{prop}[Maximal Case]\label{prop:maximal-case}
  In every genus $g \geq 2$, the (nonorientable) signatures
  \[ (3^{4g-4}) \quad\mbox{and}\quad (1,3^{4g-3}) \] 
  bound are realised.
\end{prop}
\begin{proof}
  We start with the closed case. Begin with a pants decomposition
  $P=\{a_1, \ldots, a_{3g-3}\}$, and set $Q_0 = P$. Now,
  modify $Q_0$ inductively. Let $Q_1=\{b_1,a_2,\ldots,a_{3g-3}\}$ be an adjacent
  pants decomposition in the pants complex. If
  $Q_i=\{b_1,\cdots,b_i,a_{i+1},\cdots,a_{3g-3}\}$, let
  $Q_{i+1}=\{b_1,\cdots,b_{i+1},a_{i+2},\cdots,a_{3g-3}\}$ be adjacent
  to $Q_i$. Then $P$ and $Q_{3g-3}$ are the desired multicurves. Each
  pair of pants in the complement of the curves in $P$ is cut by the
  curves in $Q$ into two hexagons and a number of quadrilaterals.

  The case of a single $1$ is proved in the same way, starting with a
  pants decomposition of a once-punctured surface, and then filling in
  the resulting puncture.
\end{proof}

\subsection{Genus 1}
We are now ready to study genus $1$. The Euler characteristic
constraint here is that the sum of $\frac{2-m_i}2$ over all
$m_i$-prongs should be $0$.

\begin{thm}\label{thm:genus1}
  Any signature in genus $1$ 
  except for $(\emptyset,-)$ and $(1,3)$ is realised
  by an optimal configuration.
\end{thm}
\begin{proof}
  The signature $(\emptyset,+)$ is realized by the meridian and the
  longitude. The other signatures we need to realize are nonorientable.
  First suppose that the list contains at least one entry $>4$. In
  that case, Proposition~\ref{prop:handle-plus-4} and
  Theorem~\ref{genus 0} imply that the list is possible in genus
  $1$. We may thus assume that every entry is $4,3$ or $1$.
		
  Next, suppose that the signature contains at least two entries equal to
  $4$. In this case Proposition~\ref{prop:handle-plus-22} and
  Theorem~\ref{genus 0} imply that the signature is realized in genus $1$.

  Suppose that the signature contains a single $4$. If it also contains a $3$,
  then Proposition~\ref{prop:handle-plus-22} and Theorem~\ref{genus 0}
  again imply that the list is possible in genus $1$. The case of
  $(1,1,4)$ we do by hand, see Figure~\ref{fig:114}.
 \begin{figure}[h]
   \includegraphics[scale=0.5]{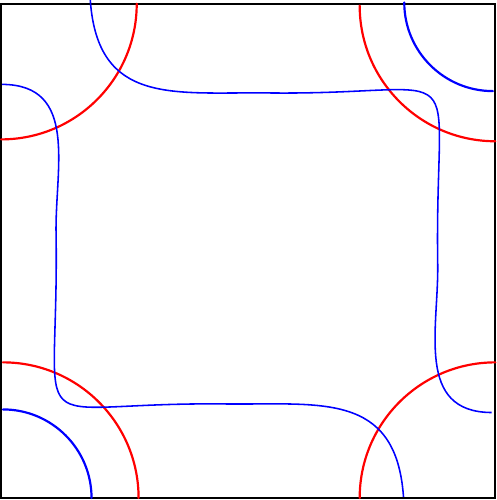}
   \caption{Realising $(1,1,4)$ on the torus}
    \label{fig:114}
  \end{figure}
  
  We are thus left with the case where the signature only contains $3$ and
  $1$ as entries. Note that both types need to appear with the same
  multiplicity by the Euler
  characteristic condition.  Since the signature $(1,3)$ is explicitly excluded,
  we will be done by induction by
  Proposition~\ref{prop:3-to-331} and Theorem~\ref{genus 0} if we can
  establish the base case that $(1,1,3,3)$ is realized. This is
  accomplished by the same procedure as in
  Proposition~\ref{prop:3-to-331} but starting with the meridian and
  the longitude, see Figure \ref{1133}.
\end{proof}

\begin{figure}[h]
   \includegraphics[scale=0.5]{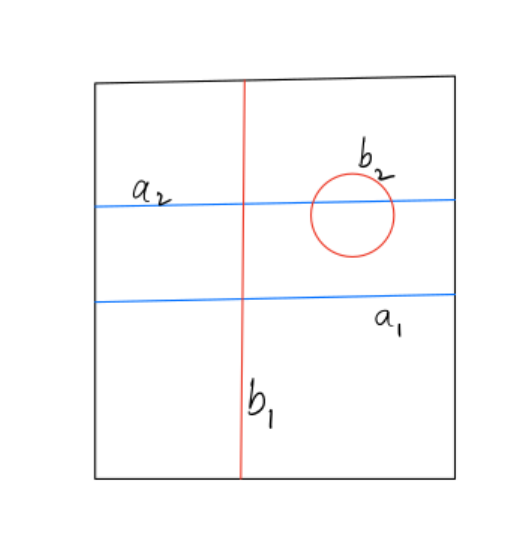}
   \caption{Realising $(1,1,3,3)$ on the torus}
    \label{1133}
  \end{figure}
        
\subsection{Genus 2}
\begin{thm}
  Every nonorientable signature in genus $2$ which satisfies the
  Euler characteristic condition except $(3,5)$ and $(6)$ is realised by an
  optimal configuration.
\end{thm}
\begin{proof}
  We argue similarly to the torus case.  First suppose that
  the signature contains at least one entry $>4$. In that case,
  Proposition~\ref{prop:handle-plus-4} and
  Theorem~\ref{thm:genus1} imply that the list is possible in
  genus $2$, unless the original list was $(3,5)$, $(6)$ or $(1,7)$. The
  first two we excluded explicitly. The last one is realised by hand in 
  Figure~\ref{71}.
  \begin{figure}[h]
    \includegraphics[scale=0.6]{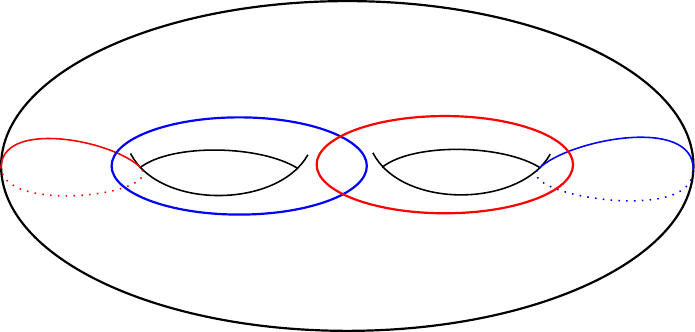}
    \caption{}
    \label{71}
  \end{figure}
  We can thus assume that the signature contains only entries $4,3$ and $1$.
  
  Next, suppose that the signature contains at least two entries equal to
  $4$. In this case Proposition~\ref{prop:handle-plus-22} and Theorem~
  \ref{thm:genus1} imply that the signature is possible in genus $1$,
  unless it is $(4,4)$ or $(1,3,4,4)$. The latter can be
  realized from $(1,1,4)$ using Proposition \ref{prop:handle-plus-22}.

  Similarly, if the signature contains both a $4$ and a $3$, then
  Proposition~\ref{prop:handle-plus-22} and Theorem~\ref{thm:genus1}
  imply that it is possible to realize it in genus $2$, unless it is
  $(3,3,4)$.

  The remaining possibility is that the signature contains only 1's and
  3's. Thus it has the form $(1^n,3^{n+4})$. 
    Proposition~\ref{prop:maximal-case} completes the argument when
    $n=0$ or $n=1$, and the case $n>1$ follows inductively from
    Proposition~\ref{prop:3-to-331}.

    It remains to realize $(4,4)$ and $(3,3,4)$.

    {\it $(3,3,4)$.} Here we view the surface of genus 2 as the double
    branched cover branched over 6 points. The multicurves are the
    preimages of the curves pictured in
    Figure \ref{334}.
\begin{figure}[h]
    \includegraphics[scale=1]{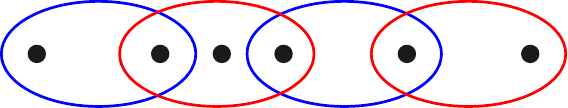}
    \caption{An optimal configuration for $(3,3,4)$ as a double
      branched cover of $S^2$ branched over 6 points.}
    \label{334}
\end{figure}
The outside component lifts to two 3's and the
region containing point 3 lifts to a 4.

{\it $(4,4)$.} Start with the meridian and the longitude on the torus
and remove a disk around the intersection point. This results in two
disjoint arcs on a genus 1 surface with boundary $b$. Now double to
get a genus 2 surface, with the curve $b$ and two curves representing
$a$. 

\end{proof}
	
For the orientable configurations, we need
\[ \sum \frac{2-m_i}{2} = -2, \] and all $m_i \geq 4$, and
therefore the only possibilities are $(4,4)$ and $(6)$ and they both
follow from Proposition \ref{prop:handle-plus-22}

\subsection{Genus 3}	
\begin{thm}
  Every signature in genus $3$ which satisfies the Euler characteristic condition is realised by an optimal configuration.
\end{thm}
\begin{proof}
  We follow our usual strategy. If at least one entry is at least $5$,
  subtracting $4$ from the largest entry works (see Proposition
  \ref{prop:handle-plus-4}) except in the cases $(5,7)$, $(3,9)$ and
  $(10)$ because that would reduce us to impossible configurations
  $(3,5)$ and $((6),-)$. The case $(5,7)$ reduces to $(1,7)$
  instead. By inspection, regions with 1 and 7 prongs are coadjacent
  in Figure \ref{71} so we can use Lemma
  \ref{lem:adding-handle-coadjacent} to reduce $(3,9)$ to
  $(1,7)$. Finally, $((10),-)$ is pictured in Figure \ref{10}.

  \begin{figure}[h]
    \includegraphics[scale=0.6]{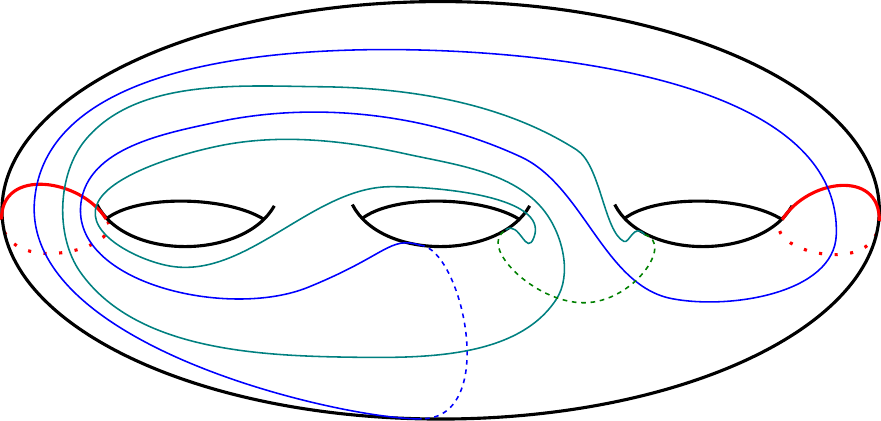}
    \caption{A nonorientable configuration with a single (10-prong) complementary component in genus $3$. The configuration is nonorientable since a blue and red curve intersect with both signs.}
    \label{10}
\end{figure}
		
  Now suppose all entries are $\leq 4$. Our reductions will not run
  into impossible configurations, since all of them have an entry
  $>4$. If there are two 4s, we can remove them, and if there is a $3$
  and a $4$, we can replace them by 1, see Proposition
  \ref{prop:handle-plus-22}. Finally, if all numbers are 1 or 3, we
  are done by Proposition~\ref{prop:maximal-case} and
  Proposition~\ref{prop:3-to-331}. 
\end{proof}

\subsection{Genus $>3$.} We argue by induction on the genus.
  As there are no impossible
  configurations in genus $3$, there are no exceptions any more in the
  previous argument.

\section{Realising Measures}
\label{sec:train-tracks}
In this section, we explain how to pass from triangular configurations
to laminations with a given number of measures. More precisely, we will prove
\begin{thm}\label{thm:matrices}
  Suppose that $((a_i, b_i)), i=1,\ldots, k$ is a optimal triangular
  configuration on a surface $S$ of genus $g$.

  Then on $S$ there is a geodesic lamination $\lambda$ which supports
  $k$ different ergodic transverse measures. Moreover, $\lambda$ is
  carried by the train track $\tau$ obtained from the multicurves by
  smoothing all intersections so that every $a_i$ turns right and
  every $b_j$ turns left, and then collapsing the bigons.

  In particular, $\lambda$ is orientable, if the triangular
  configuration is orientable.
\end{thm}
Together with the main result
Theorem~\ref{thm:existence-triangular-patterns} of the previous
section, this gives the other half of our main theorem.

\bigskip We will construct the desired train track in several
steps. Since the $a_i, b_i$ are allowed to have bigon complementary
regions and parallel curves, we begin by inserting a puncture of the
surface in every such bigon and between parallel curves, so that the
resulting curves are in minimal position on the punctured
surface. From now on we will only work on this punctured surface.

Next, smooth all intersections so that at each intersection every
$a_i$ turns right and every $b_j$ turns left. The result is a
\emph{bigon track} (see e.g. \cite[Chapter~3.4]{PH}): a smooth graph $\tau$ on the
surface. Note that $\tau$ has exactly two incoming and two outgoing
branches at each vertex, and will have (many) bigons. By construction,
the images of $\tau$ under all $D_{a_i}$ and $D_{b_j}^{-1}$ are
carried by $\tau$. 

\smallskip We now aim to compute the carrying matrices for
$D_{a_i}\tau$ and $D_{b_j}^{-1}\tau$. To do this efficiently, we note
that each branch of $\tau$ is crossed by exactly one (multi)curve.  In
other words, each $a_i$ and $b_i$ define embdedded trainpaths in
$\tau$ which do not use any common branches.  We call the union of all
branches traversed by, say $a_i$, a \emph{super-branch}, and still
denote it by $a_i$ (and similarly for $b_j$). Note that the super-branches do not necessarily generate the whole weight space; also see Lemma~\ref{lem:sb-indep} below.

In the carrying matrices, we order the superbranches as
$a_1,\cdots,a_k,b_1,\cdots,b_k$. The carrying matrices for $D_{a_i}$
and $D_{b_j}^{-1}$ preserve $\tau$, then have $2\times 2$ block form
with the blocks corresponding to the $a-$ and the $b-$ branches.

Denote by $m_{ij}$ the intersection number $I(a_i,b_j)$.

The matrix of $D_{a_i}$ has
identity $I$ in the diagonal blocks, zero matrix in the $(2,1)$-block,
and the $(1,2)$ block has zeros except the $i^{th}$ row is
$m_{i1} m_{i2} \cdots m_{ii} 0 \cdots 0$. Likewise, the matrix for
$D_{b_j}^{-1}$ will have identity on the diagonal, zero matrix in the
$(1,2)$-block, and the $(2,1)$ block will be zero except for the
$j^{th}$ row, which is $0\cdots 0 m_{jj} m_{j+1,j} \cdots m_{kj}$.

Now choose positive integers
$\alpha_1,
\alpha_2,\cdots,\alpha_k,\beta_1,\beta_2,\cdots,\beta_k$. The matrix
for $D_{a_i}^{\alpha_I}$ is obtained from the one for $D_{a_i}$ by
multiplying the $(1,2)$-block by $\alpha_i$. The matrix $A$ of
$D_{a_1}^{\alpha_1}\cdots D_{a_k}^{\alpha_k}$ (note that these
commute) is obtained by adding the $(1,2)$-blocks of each
$D_{a_i}^{\alpha_i}$ (and the other 3 blocks are still $I$ and
$0$). Thus
\[
A=\begin{pmatrix}
  I & A'\\
  0 & I
\end{pmatrix}
\]

where

\[
A'=\begin{pmatrix}
\alpha_1 m_{11} & 0 & 0 & \cdots\\
\alpha_2 m_{21} & \alpha_2 m_{22} & 0 & \cdots\\
\alpha_3 m_{31} & \alpha_3 m_{32} & \alpha_3 m_{33} & \cdots \\
\cdots & & &
\end{pmatrix}
\]

The matrix $B$ of $D_{b_1}^{\beta_1}\cdots D_{b_k}^{\beta_k}$ is
obtained similarly by adding the $\beta_j$-multiples of the
$(2,1)$-block, i.e.
\[
B=\begin{pmatrix}
I & 0\\
B' & I\end{pmatrix}
\]

where

\[
B'=\begin{pmatrix}
\beta_1 m_{11} & \beta_1 m_{21} & \beta_1 m_{31} & \cdots\\
0 & \beta_2 m_{22} & \beta_2 m_{32} & \cdots\\
0 & 0 & \beta_3 m_{33} & \cdots \\
\cdots &&&
\end{pmatrix}
\]

Finally, $AB$ is the matrix
\[
AB=\begin{pmatrix}
I+A'B' & A'\\
B' & I
\end{pmatrix}
\]

and

\[
A'B'=\begin{pmatrix}
\alpha_1\beta_1m_{11}^2 & \alpha_1\beta_1 m_{11}m_{21} &
\alpha_1\beta_1 m_{11}m_{31} & \cdots\\
\alpha_2\beta_1 m_{11}m_{21} & \alpha_2\beta_1
m_{21}^2+\alpha_2\beta_2m_{22}^2 &
\alpha_2\beta_1m_{21}m_{31}+\alpha_2\beta_2m_{22}m_{32} & \cdots\\
\alpha_3\beta_1m_{11}m_{31} & \alpha_3\beta_1
m_{21}m_{31}+\alpha_3\beta_2m_{22}m_{32} &
\alpha_3\beta_1m_{31}^2+\alpha_3\beta_2m_{32}^2+
\alpha_3\beta_3m_{33}^2 & \cdots\\
\cdots &&&
\end{pmatrix}
\]

The following definition captures the essential form of these matrices, together with properties we will need for non-unique ergodicity.

\begin{defin}
  A $2k\times 2k$-matrix $C$ with nonnegative entries is an $\epsilon$-matrix if
  \begin{enumerate}
    \item The first $k$ diagonal entries are positive and
      $c_{i+1,i+1}/c_{ii}<\epsilon$ for $i=1,2,\cdots,k-1$,
      \item $c_{ij}/c_{jj}<\epsilon$ for all $1\leq j\leq k$ and
        $i\neq j$, and
      \item $c_{ij}/c_{kk}<\epsilon$ for all $k<j\leq 2k$.
  \end{enumerate}
\end{defin}

In other words, the first $k$ columns are dominated by their diagonal
entries; these entries are quickly decreasing on the diagonal, and the
last one $c_{kk}$ still dominates all entries in the right half of the
matrix.

\begin{prop}
  For every $\epsilon>0$ we can choose $\alpha_i,\beta_j$ so that the
  matrix $AB$ is an $\epsilon$-matrix.
\end{prop}

\begin{proof}
  Fix a large number $N$ and choose the $\alpha_i,\beta_j$ so that
  \begin{itemize}
    \item
      $\alpha_1\gg\alpha_2\gg\cdots\gg\alpha_k$,
    \item $\beta_1\ll\beta_2\ll\cdots\ll\beta_k$, and
      \item
        $\alpha_1\beta_1\ll\alpha_2\beta_2\ll\cdots\ll\alpha_k\beta_k$.
  \end{itemize}
  where $P\ll Q$ means that $Q/P>N$. E.g. first choose the $\alpha_i$
  and then $\beta_j$. An inspection then shows that when $N$ is
  sufficiently large, $AB$ is an $\epsilon$-matrix.
\end{proof}

For any choice of $\alpha_i,\beta_j$ the
composition of Dehn twists considered above can be viewed as a map
$\tau\to\tau$ and it can be realized by a sequence of splittings
starting at $\tau_0:=\tau$ and ending at a bigon track $\tau_1$ which is
combinatorially isomorphic to $\tau$ but has a different marking. Now
choose a different set of $\alpha_i,\beta_j$ and consider the induced
splitting sequence from $\tau_1$ to a train track $\tau_2$ etc. This
sequence will converge to a lamination $\Lambda$.

The following is a special case of a theorem proved in \cite{EMS};
similar ideas are used in \cite[Section 8.3]{Yocsurv}.
\begin{thm}
  Alice and Bob play the following game. Alice chooses
  $\epsilon_1>0$. Then Bob chooses a splitting sequence
  $\tau_1\to\tau_0$ whose transition matrix is an
  $\epsilon_1$-matrix. Then Alice chooses $\epsilon_2>0$ and then Bob
  chooses $\tau_3\to\tau_2$ so that the transition matrix is an
  $\epsilon_2$-matrix, etc. Alice wins
  if %
  the image of infinite matrix product of the matrices Alice chose
  contains the cone over a $(k-1)$--dimensional simplex, and otherwise Bob
  wins.

  Then Alice has a winning strategy.
\end{thm}

The previous theorem implies that we can find a carrying sequence
$\tau_n$ of images of the bigon track $\tau$ under suitable twist
powers, so that the carrying matrices satisfy the conclusion of the
theorem. If $\tau$ were an actual train track (not a bigon track),
this would immediately Theorem~\ref{thm:matrices}, since a lamination
carried by every $\tau_n$ would then admit at least $k$ different
ergodic measures. Since the theory of bigon tracks is less
well-developed, we reduce to the case of train tracks using
the following lemma, to formally conclude the theorem.
\begin{lemma}\label{lem:sb-indep}
  Let $\tau$ be the bigon track constructed above.
  \begin{enumerate}
  \item By collapsing bigons in $\tau$ we obtain a train track
    $\hat{\tau}$.%
  \item The trainpaths (in $\hat{\tau}$) defined by the superbranches
    $a_i, b_j$ are independent in the weight space of $\hat{\tau}$. In
    particular, there is an isomorphism of the span of the
    superbranches in the weight space of $\tau$ with the span of their
    images in $\hat{\tau}$.
  \item Suppose that $\tau \simeq \tau' < \tau$ is a carrying relation
    inducing a matrix $M$ on the span of the superbranches, and that
    $\hat{\tau}, \hat{\tau}'$ are the corresponding train tracks from
    (1). Then, identifying the spans of superbranches as in (2), the
    corresponding carrying
    $\hat{\tau} \simeq \hat{\tau}' < \hat{\tau}$ also has the matrix $M$.
  \end{enumerate}
\end{lemma}
\begin{proof}
  \begin{enumerate}
  \item Consider the surface $X = S-\cup_i a_i$. Observe that every
    component of $X$ has negative Euler characteristic, since the
    $a_i$ are pairwise disjoint and nonisotopic, and $X$ is assumed to
    be a hyperbolic surface.

    The intersection of the union $\cup_jb_j$ with $X$ is a union
    $\Lambda$ of disjoint arcs, and every bigon of $\tau$ corresponds
    to a rectangle in $X-\Lambda$ with two sides in some $a_i$ and two
    sides in some $b_j$ (since we assume that the $a_i$ and $b_j$
    intersect minimally, and there therefore cannot be bigons formed
    by single arcs in $a_i$ and $b_j$). By the remark above, in every
    component of $X$, not every complementary component of $\Lambda$
    is such a rectangle.

    Hence, we can collapse all bigons of $\tau$ to obtain a train
    track $\hat{\tau}$. It contains the $a_i$ as embedded curves, and
    a branch for each isotopy class of arc of $\Lambda$.

  \item By construction, the $a_i$ are disjoint embedded curves in
    $\hat{\tau}$, so they are independent. Next, it is clear that no
    sum of $a_i$ and $b_i$ (with nonzero coefficients) is equal to a
    sum of only $a_i$. It remains to show that two sums which both
    have nonzero coefficients in front of some $b_i$ cannot be
    equal. Suppose this would be false, i.e
    \[ \sum (\alpha_i [a_i] + \beta_i [b_i]) = \sum (\alpha_i' [a_i] +
      \beta_i' [b_i]) \] where square brackets denote the induced
    trainpaths on $\hat{\tau}$. Let $k$ be the largest index so that
    $\beta_k \neq 0$; we may assume that $\beta'_k=0$ (by swapping the
    sides and subtracting a suitable multiple of $[\beta_k]$).  Now,
    since $b_i$ is disjoint from all $a_i, i < k$, the trainpath given
    by the right-hand-side has intersection $0$ with $a_i$. On the
    other hand, since $a_k$ and $b_k$ do intersect, the left hand side
    defines a multicurve with nonzero intersection with $b_k$. Hence
    they cannot be the same weights, as they would then define
    isotopic curves.
  \item This is immediate from (1) and (2).
  \end{enumerate}
\end{proof}
As a consequence of the lemma, we see that the resulting splitting
sequence of $\hat{\tau}$ has polyhedra of measures intersecting in a
simplex of dimension at least $k$, which shows the theorem.

\section{Upper triangular pattern of curves}
\label{sec:conf-from-fol}
In this section we show that every set of ergodic measures on the same minimal and filling topological foliation arises from the construction in the last section. 
\begin{thm}Let $\mathcal{F}$ be a minimal and filling topological
  measured foliation
  with mutually singular ergodic transverse measures $\mu_1,...,\mu_k$ and $D_*$
 be a metric giving the usual topology on $\mathcal{PMF}$. For any $\epsilon>0$ 
there exist curves $a_1,...,a_k; b_1,...,b_k$ so that 
\begin{enumerate}
\item $a_1,...,a_k,b_1,..,b_k$ are all homotopically nontrivial and
  pairwise non-homotopic.
\item $D_*(a_i,\mu_i)<\epsilon$ and $D_*(b_i,\mu_i)<\epsilon$.
\item $I(a_i,a_j)=0$ and $I(b_i,b_j)=0$ for all $i,j$. 
\item $I(a_i,b_j)=0$ for all $i<j$. 
\item $I(a_i,b_i)\neq 0$ for all $i$.
\end{enumerate}
\end{thm}

The proof occupies the remainder of this section. Roughly speaking,
the idea is to apply the Birkhoff ergodic theorem to find leaves
generic for each $\mu_i$, follow them for a long time, and then close
them up to produce the desired curves. A complication is that we will
need a quantitative version of the ergodic theorem, especially to
produce the second set of curves. 

{\bf Step 1: the setup.}
Fix an interval $I$ transverse to the foliation and let $\mu_1,...,\mu_k$ be the ergodic transverse
 measures restricted to $I$, normalized to be probability measures. 
Let $\lambda=\sum \mu_j$; we think
of $\lambda$ as the Lebesgue measure on $I$ with $\lambda(I)=k$. To
apply the ergodic theorem we construct a dynamical system by following
the leaves and keeping track of the direction they cross $I$. Thus we
let $\hat I=I\times\{+,-\}$ and we have the first return map $\hat T:\hat
I\to\hat I$ (essentially an IET). Let $\pi:\hat I\to I$ be the projection, and let
$\hat\mu_i$ be the lift of $\mu_i$, in the sense that
$\pi|I\times\{\pm\}\to I$ pushes forward the restriction of
$\hat\mu_i$ to $\mu_i$. Thus $\hat\lambda=\sum\hat\mu_j$ is a lift of
$\lambda$ and $\hat\lambda(\hat I)=2k$, $\hat\mu_i(\hat I)=2$. Let $\iota:\hat I\to\hat I$
denote the involution $(x,\pm)\mapsto (x,\mp)$. Then each $\hat\mu_i$
is $\iota$-invariant and it can be written as a finite sum of ergodic
measures. These component measures will be permuted by $\iota$ and
there is only one $\iota$-orbit since $\mu_i$ is ergodic. It
follows that $\hat\mu_i$ is either ergodic or it is the sum of two
ergodic measures interchanged by $\iota$ each of which projects to
$\mu_i$. 

{\bf Step 2: quantifying Birkhoff.}
We start by defining a subset $\hat B_i\subset\hat I$ that has a full
$\hat\mu_i$-measure and consists of points generic for
$\hat\mu_i$. When $\hat\mu_i$ is ergodic we let
$$\hat B_i=\left\{\hat x\in\hat I \,\left\vert\,\lim_{n\to\infty} \frac 1n
\sum_{j=1}^n \delta_{T^j\hat x}=\hat \mu_i\right.\right\}$$ where $\delta$ denotes
the point mass and convergence is in weak-$*$ toplogy. By Birkhoff's ergodic theorem 
this is
a $\hat\mu_i$-full measure subset of $\hat I$.

If $\hat\mu_i=\nu+\iota(\nu)$ is the sum of two ergodic measures, we
define $\hat B_i$ as the union of the sets of points generic for $\nu$
and for $\iota(\nu)$. 

Either way,
by replacing $\hat B_i$ with
$\hat B_i\cap \iota(\hat B_i)$ we may assume that it is
$\iota$-invariant, and thus $\hat B_i=\pi^{-1}(B_i)$ for some
$B_i\subset I$. Thus the $B_i$'s are pairwise disjoint and
$\mu_i(B_j)=\delta_{ij}$.

We now refine these sets by quantifying the convergence. Fix a metric
on the space of Borel measures on $\hat I$ with total mass $\leq 2k$
with weak-$*$ topology. Given any $\delta>0$ there is $A_i\subset B_i$
and $N$ such that
\begin{itemize}
\item if $\hat x\in \hat A_i:=\pi^{-1}(A_i)$ then for $n>N$
  $$\frac 1n \sum_{j=1}^n \delta_{T^j\hat x}$$
  is $\delta$-close to $\hat\mu_i/2$ when $\hat\mu_i$ is ergodic, or to
  one of the two ergodic components when it is not.
\item $\lambda(B_i\smallsetminus A_i)$ is arbitrarily small.
\end{itemize}
From now on, $\delta,N,A_i$ will be fixed, with $\delta>0$ very
small.

{\bf Step 3: quantifying Lebesgue density.}
Recall that almost every point of $A_i$ is a point of density for
$A_i$. For $\epsilon>0$ there is a subset $A_i'\subset A_i$ and $r>0$
so that
\begin{itemize}
  \item if $x\in A_i'$ and $J\subset I$ is
    an interval of length $<r$ that contains $x$ then $\lambda(J\cap
    A_i)>(1-\epsilon)\lambda(J)$.
  \item $\lambda(A_i\smallsetminus A_i')$ is arbitrarily small.
\end{itemize}

{\bf Step 4: intervals $J_i$.}
We choose points of density $x_i\in A_i'$ and a small interval $J_i$
around each $x_i$. We require that these intervals are pairwise
disjoint, that $\lambda(J_i\cap A_i')>(1-\eta)\lambda(J_i)$ for a small
$\eta>0$, and that the first $N$ iterations of $\hat
J_i:=\pi^{-1}(J_i)$ by $\hat T$ are pairwise disjoint.

{\bf Step 5: constructing $a_i$'s.}
Let $\hat T_J$ be the first return map of $\hat T$ restricted to $\cup
\hat J_i$, where $\hat J_i=\pi^{-1}(J_i)$. Then there will be $\hat y_i\in
\hat J_i$ so that $\hat y_i,\hat T_J\hat y_i, \hat T_J^2\hat y_i$ are
all in $\hat J_i\cap 
\hat A_i'$. This is because the $\hat \mu_i$-measue of $J_j$ is much
smaller than the $\hat\mu_i$-measure of $J_i\cap A_i'$, for $j\neq
i$. If two consecutive points of these three have the same second
coordinate, i.e. if the intersections with $I$ occur in the same
transverse direction, we form $a_i$ by taking the corresponding leaf
segment and closing it up with a subinterval $H_i\subset
J_i$. Otherwise, we perform the standard surgery where $a_i$ is the
union of two leaf segments and perturbations of two subintervals of
the segment $H_i\subset J_i$ that connects $\pi(\hat y_i)$ and
$\pi(\hat T_J^2\hat y_i)$.
\begin{figure}
\includegraphics[scale=0.7]{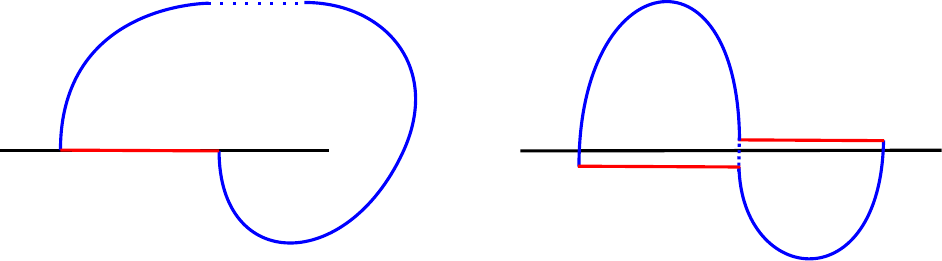}
\caption{Surgering a curve from consecutive intersections.}
\label{surgery_curves}
\end{figure}

{\bf Step 6: constructing $b_i$'s.}
After reordering, we may assume that
$$\mu_1(H_1)\geq \mu_2(H_2)\geq\cdots\geq\mu_k(H_k)>0$$
Since $H_i$ is a small interval with endpoints in $A_i'$ it follows
that the ratio $\frac{\mu_j(H_i)}{\mu_i(H_i\cap A_i)}$ is as small as we like
for $j\neq i$. Because of the reordering we then have that
the ratio
$$\frac{\mu_i(\cup_{j>i} H_j)}{\mu_i(H_i\cap A_i)}$$ is as small as we
like. We can then consider the first return map $\hat T_{H_i}$ on
$\cup_{j\geq i}\hat H_j$ (where $\hat H_j=\pi^{-1}(H_j)$) to find
$\hat z_i\in \hat H_i$ so that $\hat z_i,\hat T_{H_i}\hat z_i, \hat
T_{H_i}^2\hat z_i$ are all in $\hat H_i\cap\hat A_i$. Then construct
$b_i$ by surgery as before.

{\bf Step 7: $a_i$ and $b_i$ approximate $\mu_i$ in $PML$.}  Let
$\gamma$ be any simple closed curve in minimal position with respect
to the foliation $\mathcal{F}$.  We can isotope $\gamma$ to $\gamma'$
where it is a finite union of (horizontal) line segments in $I$ and
(vertical) segments of leaves of $\mathcal{F}$ and so that
$$I(\gamma',\mu_i)=I(\gamma,\mu_i).$$ There exists a uniform constant
$C=C_{\gamma'}$ so that if $n_{a_i}$ is the number of times
that the vertical leaf of our curve $a_i$ crosses the horizontal
segments in $\gamma'$ then
$$|I(a_i,\gamma')-n_{a_i}|\leq C.$$ Indeed, the vertical segments
$\gamma'$ do not intersect the vertical segment(s) of $a_i$ and similarly
for their horizontal segments. The horizontal segment(s) of $a_i$
intersects the vertical segments of $\gamma'$ at most the
number of times the vertical segments of $\gamma'$ intersect $I$. This
same argument shows that
$$|I(b_i,\gamma')-n_{b_i}|\leq C.$$ Applying this argument to a pair
of curves, we see that if $\delta$ is small enough then the ratio of
these curves intersection with $a_i$ can be made as close as we want
to the ratio of their intersections with $\mu_i$ and analogously for
$b_i$.

\bibliographystyle{alpha}
\bibliography{ref-ergodic}

\begin{thebibliography}{BFK24}

\bibitem[BFK24]{EMS}
Mladen Bestvina, Elizabeth Field, and Sanghoon Kwak.
\newblock Nonunique ergodicity on the boundary of outer space, 2024.

\bibitem[BW17]{BW}
Francis Bonahon and Helen Wong.
\newblock Representations of the {Kauffman} bracket skein algebra. {II}:
  {Punctured} surfaces.
\newblock {\em Algebr. Geom. Topol.}, 17(6):3399--3434, 2017.

\bibitem[CM15]{CM}
Jon Chaika and Howard Masur.
\newblock There exists an interval exchange with a non-ergodic generic measure.
\newblock {\em J. Mod. Dyn.}, 9:289--304, 2015.

\bibitem[DF22]{DF}
Michael Damron and Jon Fickenscher.
\newblock The number of ergodic measures for transitive subshifts under the
  regular bispecial condition.
\newblock {\em Ergodic Theory Dynam. Systems}, 42(1):86--140, 2022.

\bibitem[Fic14]{Fic}
Jonathan Fickenscher.
\newblock Self-inverses, {L}agrangian permutations and minimal interval
  exchange transformations with many ergodic measures.
\newblock {\em Commun. Contemp. Math.}, 16(1):1350019, 51, 2014.

\bibitem[Gab09]{Gab}
David Gabai.
\newblock Almost filling laminations and the connectivity of ending lamination
  space.
\newblock {\em Geom. Topol.}, 13(2):1017--1041, 2009.

\bibitem[Kat73]{Katok}
A.~B. Katok.
\newblock Invariant measures of flows on orientable surfaces.
\newblock {\em Dokl. Akad. Nauk SSSR}, 211:775--778, 1973.

\bibitem[Lev83]{levitt}
Gilbert Levitt.
\newblock Feuilletages des surfaces.
\newblock These de doctorat d' etat, 1983.

\bibitem[LM10]{LM}
Anna Lenzhen and Howard Masur.
\newblock Criteria for the divergence of pairs of {T}eichm\"uller geodesics.
\newblock {\em Geom. Dedicata}, 144:191--210, 2010.

\bibitem[Mas22]{MasGen}
Howard Masur.
\newblock Generic measures for translation surface flows.
\newblock {\em J. Mod. Dyn.}, 18:495--521, 2022.

\bibitem[MS93]{MS}
Howard Masur and John Smillie.
\newblock Quadratic differentials with prescribed singularities and
  pseudo-{A}nosov diffeomorphisms.
\newblock {\em Comment. Math. Helv.}, 68(2):289--307, 1993.

\bibitem[PH92]{PH}
R.~C. Penner and J.~L. Harer.
\newblock {\em Combinatorics of train tracks}, volume 125 of {\em Ann. Math.
  Stud.}
\newblock Princeton, NJ: Princeton University Press, 1992.

\bibitem[Sat75]{Sat}
E.~A. Sataev.
\newblock The number of invariant measures for flows on orientable surfaces.
\newblock {\em Izv. Akad. Nauk SSSR Ser. Mat.}, 39(4):860--878, 1975.

\bibitem[Vee69]{Vskew}
William~A. Veech.
\newblock Strict ergodicity in zero dimensional dynamical systems and the
  {K}ronecker-{W}eyl theorem {${\rm mod}\ 2$}.
\newblock {\em Trans. Amer. Math. Soc.}, 140:1--33, 1969.

\bibitem[Vee78]{VIET}
William~A. Veech.
\newblock Interval exchange transformations.
\newblock {\em J. Analyse Math.}, 33:222--272, 1978.

\bibitem[Wri15]{CylDef}
Alex Wright.
\newblock Cylinder deformations in orbit closures of translation surfaces.
\newblock {\em Geom. Topol.}, 19(1):413--438, 2015.

\bibitem[Yoc10]{Yocsurv}
Jean-Christophe Yoccoz.
\newblock Interval exchange maps and translation surfaces.
\newblock In {\em Homogeneous flows, moduli spaces and arithmetic}, volume~10
  of {\em Clay Math. Proc.}, pages 1--69. Amer. Math. Soc., Providence, RI,
  2010.

\end{thebibliography}

\end{document}